\documentclass[11pt, twoside, leqno]{article}
\usepackage{amsmath,amsthm}
\usepackage{amssymb,latexsym}
\usepackage{enumerate}

\newtheorem{theorem}{Theorem}
\newtheorem{lemma}{Lemma}

\newtheorem*{pro}{Proposition}
\newtheorem*{cor}{Corrollary}
\newtheorem*{lem}{Lemma}
\newtheorem*{lemA}{Lemma A}
\newtheorem*{lemB}{Lemma B}
\newtheorem*{sublem}{Sublemma}

\theoremstyle{definition}

\newtheorem{rema}{Remark}

\newtheorem*{rem}{Remark}

\newtheorem*{dfn}{Definition}

\numberwithin{equation}{section}

\allowdisplaybreaks

\def\reg{\text{\rm reg}\,}
\def\im{\text{\rm im}\,}
\def\FrImm{\text{\rm Fr-Imm}\,}

\begin{document}

\title{Links of singularities up to regular homotopy}

\author{by\\
Atsuko Katanaga\\
Andr\'as N\'emethi and\\
Andr\'as Sz\H{u}cs}

\date{}

\maketitle

\renewcommand{\thefootnote}{}


\renewcommand{\thefootnote}{\arabic{footnote}}
\setcounter{footnote}{0}


\begin{abstract}
The abstract link $L_d$ of the complex isolated singularity $x^2 + y^2 + z^2 + v^{2d} = 0$ in
$(\mathbb C^4, 0)$ is diffeomorphic to $S^3 \times S^2$. 
We classify the embedded links of these singularities 
up to regular homotopies precomposed with diffeomorphisms of $S^3 \times S^2$.
Let us denote  by $i_d$ the inclusion $L_d \subset S^7$.
We show that for arbitrary diffeomorphisms $\varphi_d :\ S^3 \times S^2
\longrightarrow  L_d$ the compositions $i_d \circ \varphi_d$ are image regularly homotopic for two 
values  $d_1$ and $d_2$ of $d$ if and only if $d_1 \equiv d_2 \mod 2$.
\end{abstract}

\section{Introduction}
\label{sec:1}

It is well-known that the infinite number of Brieskorn equations in $\mathbb C^5$
\[
z_1^{6k - 1} + z_2^3 + z_3^2 + z_4^2 + z_5^2 = 0,
\quad (\text{intersected with } S^9 = \bigl\{\Sigma|z_i|^2 = 1\bigr\})
\]
describe the finite number of homotopy spheres.
Why do we have infinitely many equations for a finite number of homotopy spheres?
The answer was given in \cite{E-Sz}:
These equations give all the embeddings of these homotopy spheres in $S^9$ up to regular homotopy.

The present paper grew out from an attempt to investigate the analogous question for the equations
\[
x^2 + y^2 + z^2 + v^k = 0.
\tag{$*$}
\]
It was proved in \cite{K-N} that the links of the singularities $(*)$ are $S^5$ or $S^3 \times S^2 $ depending on the parity of~$k$.
Again we have infinite number of equations for both diffeomorphism types of links.
So it seems natural to pose the analogous

\noindent
{\bf Question:}
What are the differences between the links for different values of $k$ of the same parity?
Do they represent different immersions up to regular homotopy?

For $k$ odd, when the link is $S^5$, the question about the regular homotopy turns out to be trivial, since any two immersions of $S^5$ to $S^7$ are regularly homotopic.
(By Smale's result, see \cite{S1}, the set of regular homotopy classes of immersions $S^5 \longrightarrow  S^7$ can be identified with $\pi_5 (SO_7)$.
The later group is trivial by Bott's result \cite{B}.)

The situation is quite different for $k$ even.
Put $k = 2d$ and let us denote by $L_d$ the link of $(*)$ for $k = 2d$ and by $i_d$ the inclusion $L_d \hookrightarrow S^7$.
In this case the question on regular homotopy classes of $i_d$ turns out to be not well-posed.

It is true that $L_d$ is diffeomorphic to $S^3 \times S^2 $ for any $d$, but the question about the regular homotopy makes sense only after having given a concrete diffeomorphism $\varphi_d :\ S^3 \times S^2  \longrightarrow  L_d$, and only then we can ask about the regular homotopy classes
\[
i_d \circ \varphi_d :\ S^3 \times S^2  \longrightarrow  S^7.
\]
(In the case of Brieskorn equations precomposing an immersion $f:\ \Sigma^7 \longrightarrow  S^9$ with an orientation preserving self-diffeomorphism of the homotopy sphere $\Sigma^7$ does not change the regular homotopy class of the immersion~$f$.
This is not so for the manifold $S^3 \times S^2 $.)

\begin{dfn}[see \cite{P}]
Given manifolds $M$, $N$, and two immersions $f_0$ and $f_1$ from $M$ to $N$,
we say that $f_0$ and $f_1$ are {\it image-regular homotopic} if there is a self-diffeomorphism $\varphi$ of $M$ such that $f_1$ is regularly homotopic to $f_0 \circ \varphi$.
\end{dfn}

\noindent
{\bf Notation:}

1) $I(M, N)$ will denote the image-regular homotopy classes of immersions of $M$ to $N$.
The image regular homotopy class of an immersion $f$ will be denoted by $\im[f]$.

2) Recall that an immersion is called framed if its normal bundle is trivialized.
$\FrImm (M, N)$ will denote the framed regular homotopy classes of framed immersions of $M$ to~$N$.

In the case when the immersion $f$ is framed $\reg[f]$ will denote its framed regular homotopy class.

\begin{rem}
Note that for the inclusions $i_d:L_d \subset S^7$ their regular homotopy classes $\reg[i_d]$ are not well-defined, but their image regular homotopy classes  $\im[i_d]$ are well-defined.
\end{rem}

\section*{Formulation of the results}

\begin{theorem}
\label{th:1}
For any simply connected, stably parallelizable, $5$-dimensional manifold $M^5$ the framed regular homotopy classes of framed immersions in $S^7$ can be identified with $H^3(M; \mathbb Z)$, i.e.
\[
\FrImm (M^5, S^7) = H^3(M; \mathbb Z).
\]
\end{theorem}

\begin{cor}
In particular,
\[
\FrImm(S^3 \times S^2, S^7) = \mathbb Z.
\]
\end{cor}

\begin{theorem}
\label{th:2}
The set $I(S^3 \times S^2)$ of image-regular homotopy classes of framed immersions $S^3 \times S^2 \longrightarrow  S^7$ can be identified with $\mathbb Z_2$.
\end{theorem}

\begin{theorem}
\label{th:3}
The inclusions $i_d :\ L_d \hookrightarrow S^7$ for $d = d_1$ and $d_2$
represent the same element in $I(S^3\times S^2, S^7) = \mathbb Z_2$ (i.e.\ $\im[i_{d_1}] = \im[i_{d_2}]$) if and only if $d_1 \equiv d_2 \mod 2$.
\end{theorem}

\begin{rem}
The identifications in the above
Theorems
 arise only after we have fixed a
parallelization of the manifolds (or a stable parallelization). (Different
parallelizations provide different identifications.
For the Corollary these identifications differ by an affine shift
$x \mapsto x+a$, where  $a \in \pi_3(SO)= \mathbb Z$ is the difference of the
two parallelizations. Similarly, in Theorem~\ref{th:2}, $a$ is
replaced by $a$ mod $2$ in
$\mathbb Z_2 = \pi_3(SO)/\im j_*(\pi_3(SO(3)))$, where $j$ is the inclusion
$j:SO(3) \subset SO.$
Now we describe a concrete stable parallelization of $S^3 \times S^2$ we shall use.
\end{rem}

Hence, we want to choose a trivialization of the stable tangent bundle
\[
T(S^3 \times S^2) \oplus \mathcal E^1\longrightarrow S^3 \times S^2,
\]
where $\mathcal E^1$ is the trivial real line bundle.
This $6$-dimensional vector bundle is the same as the restriction $T(S^3 \times
\mathbb R^3) \raisebox{-3pt}{$\Bigr|_{S^3 \times S^2}$} = (p_1^* TS^3
\oplus p_2^* T\mathbb R^3)\raisebox{-3pt}{$\Bigr|_{S^3 \times S^2}$}$, where
$p_1$ and $p_2$ are the projections of $S^3 \times \mathbb R^3$ onto the factors.
The quaternionic multiplication on $S^3$ gives a trivialization of $TS^3$,
i.e.\ an identification with $S^3 \times \mathbb R^3$.
We need a trivialization of $T(TS^3)$.
The standard spherical metric on $S^3$ gives a connection on the bundle
$TS^3 \longrightarrow  S^3$, that is a ``horizontal'' $\mathbb R^3 \subset T(TS^3)$ at any point.
The trivialization of $TS^3$ gives a trivialization of both the horizontal and the vertical (tangent to the fibers) components in $T(TS^3)$.
Restricting this to the sphere bundle $S(TS^3) = S^3 \times S^2$ we obtain the required trivialization of
\[
T(TS^3)\raisebox{-3pt}{$\Bigr|_{S^3 \times S^2}$} = T(S^3 \times
\mathbb R^3)\raisebox{-3pt}{$\Bigr|_{S^3 \times S^2}$} = T(S^3 \times S^2) \oplus \mathcal E^1.
\]

\begin{proof}[Proof of Theorem~\ref{th:1}]
Having fixed a stable parallelization of $M^5$ any framed immersion $f:\ M^5
\longrightarrow  \mathbb R^q$ gives a map $M^5 \longrightarrow  SO_q$ that -- by a slight abuse of notation -- we will denote by $df$.

By the Smale--Hirsch immersion theory \cite{S1,H} the map
\[
\begin{array}{cll}
\FrImm(M, \mathbb R^q) & \longrightarrow & [M, SO_q]\\
\reg[f] & \longrightarrow & [df]
\end{array}
\]
induces a bijection, where $[M, SO_q]$ denotes the homotopy classes of maps $M \longrightarrow  SO_q$.

Since $M^5$ is simply connected there is a cell-decomposition having a single $0$-cell, a single $5$-cell, and no $1$-dimensional, neither $4$-dimensional cells.

Let $\overset{\circ}{M}$ be the punctured $M^5$: $\overset{\circ}{M} = M^5 \setminus D^5$.
From the Puppe sequence of the pair $(\overset{\circ}{M}, \partial \overset{\circ}{M})$ (see
\cite{Hu}),

$$S^4 = \partial \overset{\circ}{M }\subset \overset{\circ}{M} \subset M \longrightarrow S^5,$$ it follows that
the restriction map  $[M^5, SO_q] \to [\overset{\circ}{M}, SO_q]$ is a
bijection, since $\pi_4(SO)=0$ and  $\pi_5(SO_q) = 0$.

Now consider the Puppe sequence of the pair $(\overset{\circ}{M}, sk_2\, M)$.
Note that $sk_2 \, M$ is a bouquete of $2$-spheres, while the quotient $\overset{\circ}{M} / sk_2\, M$ is homotopically equivalent to a bouquette of $3$-spheres.
Hence, a part of the Puppe sequence looks like this:
\[
sk_2\, M \subset \overset{\circ}{M} \longrightarrow \vee S^3 \longrightarrow S(sk_2\, M) = \vee S^3
\]
where $S(\ )$ means the suspension.
Mapping the spaces of this Puppe sequence to $SO_q$, $q \geq 5$, we obtain the following exact sequence of groups (we omit $q$):
\[
\bigl[sk_2\, \overset{\circ}{M}, \, SO\bigr] \longleftarrow \bigl[\overset{\circ}{M}, \, SO\bigr] \longleftarrow \bigl[\vee S^3, \, SO\bigr] \overset{\alpha}{\longleftarrow} \bigl[S(sk_2\, M), SO\bigr].
\]
Here $\bigl[sk_2\, \overset{\circ}{M}, \, SO] = 0$, because $\pi_2(SO) = 0$.

Since $\pi_3(SO) = \mathbb Z$ the group $[\vee S^3, \, SO]$ can be identified with the group of $3$-dimensional cochains of $M$ with integer coefficients, i.e.\ $[\vee S^3, \, SO] = C^3(M; \mathbb Z)$.

Since there are no $4$-dimensional cells this is also the group of $3$-dimen\-sional cocycles.
The group $\bigl[S(sk_2\, M), \, SO\bigr]$ can be identified with the group of $2$-dimensional cochains $C^2(M; \mathbb Z)$.

\begin{lem}
The map $\alpha$ can be identified with the coboundary map
\[
\delta{:}\,C^2(M; \mathbb Z) \longrightarrow C^3(M; \mathbb Z).
\]
\end{lem}

Proof of this Lemma will be given in the Appendix.

Hence the cokernel of $\alpha$, i.e.\ $\bigl[\overset{\circ}{M}, SO\bigr] =
\FrImm(M, \mathbb R^q)$ can be identified with the cokernel of $\delta$, i.e.\ with $H^3(M; \mathbb Z)$.
\end{proof}

\begin{rema}
\label{rema:1}
In the case when $M = S^3 \times S^2$ and $N \in S^2$  is a fixed point
in $S^2,$ for example the North pole, the inclusion $S^3 \hookrightarrow M$,
$x \longrightarrow  (x, N)$ gives an isomorphism
\[
[M, SO] \longrightarrow [S^3, SO].
\]
Hence, for $M = S^3 \times S^2$ two framed immersions $M^5 \longrightarrow
\mathbb R^7$ (or $M^5 \longrightarrow  S^7$) are regularly homotopic if their restrictions to $S^3 \times N$ are framed regularly homotopic (adding the two normal vectors of $S^3$ in $M^5$ to the framing).
\end{rema}

\begin{lemma}
\label{lem:1}
The inclusion $j:\ SO_3 \hookrightarrow SO_q$ $(q \geq 5)$ induces in $\pi_3$
the multiplication by $2$ (if we choose the generators in $\pi_3(SO_3) =
\mathbb Z$ and in $\pi_3(SO_q) = \mathbb Z$ properly), i.e. for any $x \in \pi_3(SO_3) = \mathbb Z$ the image $j_*(x) \in \pi_3(SO) = \mathbb Z$ is $2x$.
\end{lemma}

\begin{proof}
It is well-known that $\pi_3(SO_5) \approx \pi_3(SO_6) \approx \dots \approx \pi_3(SO)$ and by Bott's result \cite{B} $\pi_3(SO) \approx \mathbb Z$.
Let us consider $V_2(\mathbb R^5) = SO_5 / SO_3$.
It is well-known that $\pi_3\bigl(V_2(\mathbb R^5)\bigr) = \mathbb Z_2$ (see for example \cite{M-S}).
It is also well-known that
\[
\pi_3(SO_3) = \mathbb Z.
\]
Now the exact sequence of the fibration $SO_5 \longrightarrow  V_2(\mathbb R^5)$ gives
that the homomorphism $\pi_3(SO_3) \longrightarrow  \pi_3(SO_5)$ induced by
the inclusion is a multiplication by $+2$ (or $-2$, but choosing the
generators properly it can be supposed that it is multiplication by $+2$).
\end{proof}

\begin{rema}
\label{rema:2}
It is well-known that
\[
\pi_3(SO_4) = \pi_3(S^3) \oplus \pi_3(SO_3)
\]
and the map $j_{4*}:\ \pi_3(SO_4) \longrightarrow  \pi_3(SO_5)$ induced by the inclusion $j_4 :\ SO_4 \hookrightarrow SO_5$ is epimorphic.

It follows that $j_{4*}$ maps $\pi_3(S^3) = \mathbb Z$ to the group $\mathbb Z_2 = \pi_3(SO_5) / j_{4*}\bigl(\pi_3(SO_3)\bigr)$ epimorphically.
\end{rema}

From now on we shall denote by $M$ the manifold $S^3 \times S^2$ (except in the Appendix).
We shall write simply $S^3$ for the subset $S^3 \times N \subset S^3 \times S^2$, where $N \in S^2$.

\begin{lemma}
\label{lem:2}
For any class $2m \in 2\mathbb Z = \im j_{4*} \subset \mathbb Z = \pi_3(SO)$
there is a diffeomorphism $\alpha_m : \ M \longrightarrow  M$ such that for
any framed immersion $f:\ M \longrightarrow  \mathbb R^7$ the difference of the regular homotopy classes of $f$ and $f \circ \alpha_m$ is $2m$, i.e.
\[
\reg[f \circ \alpha_m] - \reg[f] \in \pi_3 (SO) = \mathbb Z
\]
is $2m$.
\end{lemma}

\begin{proof}
Let $\mu_m:\ S^3 \longrightarrow  SO_3$ be a map representing the class $m \in \pi_3(SO_3)$ and define the diffeomorphism
\[
\alpha_m :\ S^3 \times S^2 \longrightarrow S^3 \times S^2
\]
by the formula
\[
(x,y) \longmapsto (x, \mu_m(x)y).
\]
We have the following diagram:
\[
\begin{array}{ccc}
\reg[f]\in \FrImm(M, \mathbb R^q) & \longrightarrow & \FrImm(S^3, \mathbb R^q) \ni  \reg\Bigl[f\raisebox{-3pt}{$\Bigr|_{S^3}$}\Bigr]\\
\downarrow\approx & & \downarrow\approx\\
{[M, SO_q]} & \longrightarrow & [S^3, SO_q] \\
df & \longmapsto & df\raisebox{-3pt}{$\Bigr|_{S^3}$}\\
d(f\circ \alpha_m) & \longmapsto & d(f\circ \alpha_m)\raisebox{-3pt}{$\Bigr|_{S^3}$}
\end{array}
\]
It shows that the regular homotopy class of the (framed) immersion $f$ is
detected by the homotopy class of $df\Bigr|_{S^3}$ in $\pi_3(SO)$, while
the regular homotopy class of $f\circ \alpha_m$ is detected by the homotopy
class of $d(f \circ \alpha_m)\Bigr|_{S^3}.$

So we have to compare the homotopy classes of maps
\[
df\raisebox{-3pt}{$\Bigr|_{S^3}$}:\ S^3 \longrightarrow SO_q \ \text{ and } \ d(f\circ \alpha_m)\raisebox{-3pt}{$\Bigr|_{S^3}$}:\ S^3 \longrightarrow SO_q.
\]
By the chain rule one has:
\[
d(f\circ \alpha_m)\raisebox{-3pt}{$\Bigr|_{S^3}$} = df\raisebox{-3pt}{$\Bigr|_{\alpha_m(S^3)}$} \cdot d \alpha_m \raisebox{-3pt}{$\Bigr|_{S^3}$}.
\]

The restriction map $\alpha_m\raisebox{-3pt}{$\Bigr|_{S^3}$}:\ S^3 \longrightarrow  S^3 \times S^2$ is homotopic to a map into $S^3 \vee S^2$, representing in the third homotopy group $\pi_3(S^3 \vee S^2) = \mathbb Z \oplus \mathbb Z$ the element $(1, *)$, where $*$ is an integer, $* \in \pi_3(S^2) = \mathbb Z$ (at this point its value is not important, but later we shall show that it is $m$, see Lemma~A).
Since the map $df$ maps $S^3 \times S^2$ into $SO$ and $\pi_2(SO) = 0$, the map $df\raisebox{-3pt}{$\Bigr|_{S^3 \vee S^2}$}$ can be extended to
\[
S^3 \vee D^3 \cong S^3.
\]

Finally we have that $d(f\circ \alpha_m)\raisebox{-3pt}{$\Bigr|_{S^3}$}$ is homotopic to the pointwise product of the maps $df\raisebox{-3pt}{$\Bigr|_{S^3}$}$ and $d\alpha_m\raisebox{-3pt}{$\Bigr|_{S^3}$}$.

But it is well-known that this gives the sum of the homotopy classes $\Bigl[df\raisebox{-3pt}{$\Bigr|_{S^3}$}\Bigr] \in \pi_3(SO)$ and $\Bigl[ d\alpha_m \raisebox{-3pt}{$\Bigr|_{S^3}$}\Bigr] \in \pi_3 (SO)$.

It remained to show the following

\begin{sublem}
$\Bigl[ d\alpha_m\raisebox{-3pt}{$\Bigr|_{S^3}$}\Bigr] = 2m \in \pi_3(SO_q) = \mathbb Z$.
\end{sublem}

\begin{proof}
The differential $d \alpha_m$ acts on $T(S^3 \times \mathbb
R^3)\raisebox{-3pt}{$\Bigr|_{S^3 \times S^2}$} = p_1^* TS^3 \oplus p_2^*
T\mathbb R^3\raisebox{-3pt}{$\Bigr|_{S^3 \times S^2}$}$ as follows:
by identity on $p_1^*TS^3$ and by $\mu_m(x)$ on $(x,y) \times \mathbb R^3$ for any $x \in S^3$, $y \in S^2$.

Hence, $d\alpha_m\raisebox{-3pt}{$\Bigr|_{S^3}$}$ is $j \circ \mu_m$, where $j:\ SO_3 \hookrightarrow SO_q$ is the inclusion.
Recall that the map $\mu_m:\ S^3 \longrightarrow  SO_3$ was chosen so that its homotopy class $[\mu_m] \in \pi_3(SO_3)$ is $m \in \mathbb Z = \pi_3(SO_3)$.
Since $j_*$ is ``the multiplication by $2$'' map it follows that $\Bigl[d\alpha_m \raisebox{-3pt}{$\Bigr|_{S^3}$}\Bigr] = 2m$.
\end{proof}
This ends the proof of Lemma \ref{lem:2} too.
\end{proof}

\begin{pro}
Any self-diffeomorphism of $S^3 \times S^2$ changes the regular homotopy class
of any immersion by adding an element of the subgroup in $\im j_* = 2\mathbb Z \subset \mathbb Z = \pi_3(SO)$.
That is for any framed immersion $f:\ M \longrightarrow  \mathbb R^q$ with (framed) regular homotopy class $\reg[f] \in [M, SO] = \pi_3(SO)$ and any diffeomorphism $\varphi : \ M \longrightarrow  M$ the difference of regular homotopy classes $\reg[f] - \reg[f \circ \varphi]$ belongs to the subgroup $\im j_* = 2\mathbb Z$ in $\mathbb Z = \pi_3(SO)$.
\end{pro}

The proof will rely on the following two lemmas (Lemma A and Lemma B).

\begin{dfn}
A self-diffeomorphism $\varphi : \ S^3 \times S^2 \longrightarrow S^3 \times S^2$ will be called \emph{positive} if it induces on $H_3(S^3 \times S^2) = \mathbb Z$ the identity.
\end{dfn}

\begin{lemA}
For any positive self-diffeomorphism $\varphi$ there exists a natural number
$m \in \mathbb Z$ such that for $N \in S^2$ the restrictions
$\varphi\raisebox{-3pt}{$\Bigr|_{(S^3 \times N)}$}$ and $\alpha_m\raisebox{-3pt}{$\Bigr|_{(S^3 \times N)}$}$ represent the same homotopy class in $\pi_3(M)$.
\end{lemA}

\begin{lemB}
Let $\varphi$ and $\psi$ be self-diffeomorphisms of $M$ such that the images of $S^3 \times N$ at $\varphi$ and $\psi$ represent the same element in $\pi_3(M)$.
Then for any framed-immersion $f:\ M \longrightarrow  \mathbb R^7$ the regular homotopy classes of $f \circ \varphi$ and $f \circ \psi$ coincide.
\end{lemB}

\begin{proof}[Proof of Lemma B]
Let us extend the self-diffeomorphisms $\varphi$ and $\psi$ to those of $M
\times D^q$ by taking the product with the identity map of $D^q$, for some
large $q$, and denote these self-diffeomorphisms of $M \times D^q$ by $\hat\varphi$ and
$\hat\psi$. Similarly we shall denote by $\hat f$ the product of $f$ with the
standard inclusion $D^q \subset \mathbb R^q.$

By the Smale--Hirsch theory \cite{S1,H} (or by the so-called Compression Theorem of Rourke--Sanderson \cite{R-S}) the restriction induces a bijection
\[
\FrImm(M, \mathbb R^7) \longleftarrow \FrImm(M \times D^q, \mathbb R^{7 + q}).
\]
Again the regular homotopy class of a framed immersion in
\[
\FrImm(M, \mathbb R^{7 + q}) = \FrImm(M \times D^q, \mathbb R^{7 + q})
\]
 is uniquely defined by the restriction to $S^3 (= S^3 \times N)$.

The maps $\hat\varphi$ and $\hat \psi$ restricted to the sphere $S^3 \times N$ are framed isotopic.
By Thom's isotopy lemma \cite{T} there is an isotopy $\Psi_t : \ M \times D^q \longrightarrow  M \times D^q$ such that $\Psi_0 = \hat\varphi$ and $\Psi_1 = \hat \psi$.

It follows that the induced maps $d\hat\varphi:\ M \longrightarrow  SO$ and $d\hat\psi :\ M \longrightarrow  SO$ are homotopic.
Hence, the framed-regular homotopy classes of $\hat f\circ \hat \varphi$ and $\hat f \circ \hat \psi$ coincide.
Then the compositions $f\circ \varphi$ and $f\circ \psi$ are also regularly homotopic.
\end{proof}

\begin{proof}[Proof of Lemma A]
Let $m$ be the homotopy class of the composition
\[
S^3 \overset{i_\varphi}{\hookrightarrow} S^3 \times S^2 \overset{p}{\longrightarrow} S^2,
\]
 where $i_\varphi$ is the inclusion $x\mapsto \varphi(x, N)$ and $p$ is the projection $S^3 \times S^2 \longrightarrow  S^2$.
We claim that the maps $\varphi'= p \circ \varphi\raisebox{-3pt}{$\Bigr|_{(S^3
    \times N)}$}$ and $\alpha_m' = p \circ
\alpha_m\raisebox{-3pt}{$\Bigr|_{(S^3 \times N)}$}$ are homotopic maps from
$S^3$ to  $ S^2$.
To show this it is enough to compute the Hopf invariants of these maps.

Let us consider first the case $m = 1$.
We need to show that the Hopf invariant of $\alpha_1'$ is equal to~$1$.

The map $\mu_1:\ S^3 \longrightarrow  SO_3$ representing the generator in
$\pi_3(SO_3)$ can be provided by  the standard double covering $S^3 \longrightarrow  SO_3$.
Then $\alpha_1$ is the self-diffeomorphism of $S^3 \times S^2$
\[
\alpha_1(x, y) = \bigl(x, \mu_1(x)y\bigr)
\]
and $\alpha_1'$ is the composition of the following three maps: the inclusion
\[
S^3 \hookrightarrow S^3 \times S^2,\quad x \longmapsto  (x, N);
\]
the map $\alpha_1$ and the projection $p:\ S^3 \times S^2 \longrightarrow  S^2$.

In order to compute the Hopf invariant of $\alpha_1':\ S^3 \longrightarrow  S^2$ first we need to compute the preimage of a regular value.
Let us compute first the preimage of $N$ in $S^3$, i.e. $(\alpha_1')^{-1}(N)$.
The map $\alpha_1'$ can be further decomposed as the composition of $\mu_1:\ S^3 \longrightarrow  SO(3)$ with the evaluation map $e:\ SO_3 \longrightarrow  S^2$, $g\mapsto g(N)$, for $g \in SO_3$.
The set $e^{-1}(N)$ is the subgroup $SO_2 \subset SO_3$, which consists of the rotations around the line $\overrightarrow{(N, -N)}$ (the stabilizer subgroup of $N$).

When we identify $SO_3$ with the ball $D_\pi^3$ of radius $\pi$ with identified antipodal points on the boundary $S_\pi^2$, then this subgroup $SO_2$ corresponds to the diameter $\overline{N, -N}$
with identified endpoints $N$ and $-N$.
The preimage of this diameter at $\mu_1:\ S^3 \longrightarrow  SO_3$ is a great circle.
If we take any other point $V$ in $S^2$, then $e^{-1}(V)$ is a coset of the previous subgroup $SO_2$.
Then its preimage at $\mu_1$ is also a great circle.
Therefore the linking number of two such preimages is~$1$.

The map $\alpha_m'$ can be obtained from $\alpha_1'$ by precomposing it with a degree $m$ map $S^3 \longrightarrow  S^3$.
Hence the Hopf invariant of $\alpha_m'$ is $m$.
\end{proof}

\section*{Parametrizations of the links $L_d$ (or, equivalently, of the singularities $X_d$)}
Let us denote by $\zeta$ the complex $\mathbb C^2$-bundle $T\mathbb CP^1 \oplus
\varepsilon^1_C$ over $\mathbb CP^1 = S^2$, where $T\mathbb CP^1$ is the tangent
bundle of $\mathbb CP^1$, and $\varepsilon^1_C$ is the trivial complex line
bundle. Note that the bundle  $\zeta$ considered as a real $\mathbb
R^4$-bundle is isomorphic to the trivial bundle. Hence its total space is
diffeomorphic to $S^2\times \mathbb R^4 .$ Let us denote by $E_0(\zeta)$ the
complement of the zero section in the total space of the bundle $\zeta.$
We shall give below a {\it {diffeomorphism}} of this  space $E_0(\zeta)$ onto $X_d \setminus 0.$
The existence of such a diffeomorphism will give a new proof of the
result of [KN] about the diffeomorphism type of $L_d.$
\begin{pro}
$L_d$ is diffeomorphic to $S^3\times S^2.$
\end{pro}
\begin{proof}
$X_d\setminus 0$ is diffeomorphic to $\L_d \times \mathbb R^1,$
and the space $E_0(\zeta)$ is diffeomorphic to  $S^3 \times S^2 \times \mathbb
R^1.$ For simply connected $5$-manifolds it is well-known, that two such
manifolds  are
diffeomorphic if their products with the real line are diffeomorphic (see
\cite {Ba}, Theorem 2.2).
Hence $L_d$ and $S^3 \times S^2$ are diffeomorphic.
\end{proof}
Next we give a concrete parametrization:
\[
\aligned
&\varphi_d :
\  E_0(\zeta) \longrightarrow X_d \setminus 0 \\
&= \bigl\{x, y, z, v\ \big|\  x^2 + y^2 + z^2 + v^{2d} = 0, \ |x| + |y| + |z| + |v| \neq 0 \bigr\}.
\endaligned
\]
The composition $i_d \circ \varphi_d$ (or its restriction to $\varphi_d^{-1}(S^7)$) will give a framed-immersion $S^3 \times S^2 \longrightarrow  S^7$, and its regular homotopy class $\reg[i_d \circ \varphi_d]$ will turn out to be the number $d \in \mathbb Z = \FrImm(S^3 \times S^2, S^7)$.

This will imply that the image-regular homotopy class of the link $L_d$ in $S^7$ is $d\mod 2$ in $\mathbb Z_2 = I(S^3 \times S^2, S^7)$.

\begin{proof}
For arbitrary manifolds $N$ and $Q$ the natural map
\[
\FrImm(N, Q) \longrightarrow  \FrImm(N, Q \times \mathbb R^1)
\]
 induces a bijection -- by the Smale--Hirsch immersion theory (or by the Compression Theorem of Rourke--Sanderson).
Hence $\FrImm(X_d \setminus 0 \subset \mathbb C^4 \setminus 0) = \FrImm (S^3 \times S^2 \subset S^7)$.
By a coordinate transformation of $\mathbb C^4$ we obtain the following equivalent equation defining $X_d$
\[
X_d = \bigl\{ x, y, z, v\ \big|\ xy - z(z + v^d) = 0\bigr\}.
\]

The parametrization of $X_d \setminus 0$ is the following.

The inclusion
\[
 E_0(\zeta) \overset{\Psi}{\longrightarrow}
\mathbb C^4 = \{(x, y, z, v) \mid  x, y, z, v \in \mathbb C\} \text{ with image } \ \im \Psi = X_d \setminus 0
\]
will be described on two charts:

1) $\bigl((a : b), x, v\bigr)$, where $a, b, x, v \in \mathbb C$, $b \neq 0$,
$(a:b) \in \mathbb  CP^1$, and\break
$\|x\| + \|v\| \neq 0$.
Put $t = \frac{a}{b} \in \mathbb C$.

The map $\Psi$ on this chart will be given by the formula
\[
\Psi:\ (t, x, v) \longrightarrow (x, t^2 x + tv^d , tx, v).
\]

2) For $a \neq 0$ denote the quotient $\frac{b}{a}$ by $t'$.
On the  part of $E_0(\zeta)$ that projects to $\mathbb CP^1 \setminus (1:0)$
(that is diffeomorphic to $\mathbb CP^1 \setminus (1:0) \times (\mathbb C^2 \setminus 0))$ consider the coordinates $(t', y, v)$ and define $\Psi$ by the formula
\[
\Psi:\ (t', y, v) \longrightarrow (t^{\prime 2} y - t'v^d, y, t'y - v^d, v).
\]
The change of coordinates between the two coordinate charts of $E_0(\zeta)$ is
\[
\aligned
t'= t^{-1}, \ v = v, \ x &= t^{\prime 2} y - t'v^d \ \text{ or equivalently}\\
y &= t^2 x + tv^d . \
\endaligned
\]
In order to see that these local coordinates give indeed the bundle $\zeta$
over $\mathbb CP^1$ we can precompose the first local system with the map
$(t, x,v) \mapsto (t, x-tv^d,v).$ (Note that this map can be connected to the
identity by the diffeotopy $(t,x,v) \mapsto (t, x-stv^d, v)$,\ $0 \le s \le 1.)$ Then the change from the first coordinate
 system to
the second one for $t \in S^1$ on the equator of $S^2 = \mathbb CP^1$ will be
given by the map $(t,x,v) \mapsto (t, t^2x,v)$, where $x,v \in \mathbb C.$
Now it is clear that the obtained bundle is $\zeta = T\mathbb CP^1 \oplus \varepsilon^1_C.$
(The  map of the equator to $U(2)$ defining the bundle $\zeta$ gives in $\pi_1(U(2))$ the double of the
generator, and its image in $\pi_1(SO(4) = Z_2$  is trivial. That is why the
bundle $\zeta$ is trivial as a real bundle although it has first Chern class
equal $2$ as a complex bundle.)
Note that $\Psi$ maps the part of the first chart corresponding to the points
$t=0$, (i.e. the space  $\mathbb C^2 = \{(0:1), x,v\}$) identically onto the coordinate
space $\mathbb C^2_{x,v} =\{x,0,0,v\}$ of $\mathbb C^4.$
The restriction of $\Psi$ to this part determines the framed immersion of
$X_d\setminus 0$ to $\mathbb C^4.$ Hence, the immersion itself is very
simple: just the inclusion of $\mathbb C^2\setminus 0 \to \mathbb C^4.$ But we need  to
consider also the framing. It is coming
a) from the paramatrization $\Psi$ and b) from the defining equation of $X_d.$
\ \

a) The parametrization gives the complex vector field
\[
\frac{\partial \Psi}{\partial t} \raisebox{-4pt}{$\Bigr|_{t = 0}$} = (0, v^d, x, 0).
\]
b) The defining equation
$g(x,y,z,v) = xy - z(z + v^d) = 0$ at the points $(x,0,0,v)$
gives the complex vector field
\[
\text{\rm grad}\, g (x,0,0,v) = (0, x, - v^d, 0).
\]

These two complex vector fields have zero first and last complex coordinates
(on the coordinate subspace $\mathbb C^2_{x,v} =\{x,0,0,v\}$).
Hence, we shall write only their second and third coordinates: those are
$(v^d, x)$ and $(-x,v^d)$ respectively. These two complex vectors
 give four real vector fields if we add their $i$-images as well.
Let us denote by $a_1$ and $a_2$  the real and imaginary coordinates of $v^d$: $v^d = a_1 + ia_2$.
Similarly $x_1$ and $x_2$ are those of $x$, i.e. $x = x_1 + i x_2$.
Then the four real vectors in $\mathbb R^4 =\mathbb C^2 =
(0,y,z,0)$ are:
\begin{align*}
\bold u_1 &= (a_1, a_2,  x_1, x_2) \ \text{ }\\
\bold u_2 &= (a_2, -a_1, x_2, -x_1) \ \text{ }\\
\bold u_3 &= (x_1, x_2, - a_1, -a_2) \ \text{ }\\
\bold u_4 &= (-x_2, x_1, a_2, -a_1).
\end{align*}

The map
$
(x, v) \in \mathbb R^4\setminus 0 \longrightarrow  (\bold u_1, \bold u_2, \bold u_3, \bold u_4)
$
can be decomposed as a degree $d$ branched covering $(x,v) \mapsto (x,v^d)$
and a map representing an element in $\pi_3(SO_4) = \pi_3(S^3) \oplus \pi_3(SO_3)$
of the form $(1, *)$ for some unknown element $*$ in  $\pi_3(SO_3).$
(This is because the map $(x,v^d)= (x_1,x_2, a_1, a_2) \mapsto \bold u_1 =
(a_1, a_2, x_1, x_2)$ is almost the identity,
it differs only by an even permutation of the coordinates.)
Hence the composition represents an element of the form $(d, ?)\in \pi_3(S^3)
\oplus \pi_3(SO_3)$, and its image
in $\pi_3(SO)/j_{4*}(\pi_3(SO_3) = \mathbb Z_2$ is $d \,\text{\rm mod}\,2$, see
Remark 2.
That finishes the proof of Theorem 3.
\end{proof}

\section*{Appendix}

For any space $Y$ let us denote by $CY$ the cone over $Y$.
Here we show that the map provided by the Puppe sequence
\[
\alpha:\ \bigl[C(\overset{\circ}{M}) \cup C(sk_2 M), SO\bigr] \longrightarrow  \bigl[ \overset{\circ}{M} \cup C(sk_2 M), SO\bigr]
\]
can be identified with the coboundary map in the cochain complex:
\[
\delta:\ C^2(M; \mathbb Z) \longrightarrow  C^3 (M; \mathbb Z).
\]
We have seen that the sources and targets of $\delta$ and $\alpha$ can be identified.

For simplicity let us consider the situation when $sk_2 M = S^2$ and $M$ has a single $3$-cell $D^3$, attached to this $S^2$ by a map $\theta$ of degree~$k$.
Then $\overset{\circ}{M} = S^2 \underset{\theta}{\cup} D^3$.

Let us denote the sets
\[
\overset{\circ}{M} \cup C(sk_2 M) \ \text{ and } \ C\overset{\circ}{M} \cup C(sk_2 M)
\]
by $A$ and $B$ respectively.

Clearly we can choose \emph{any} degree $k$ map for $\theta$ in order to study the induced map~$\alpha$.
Take for $\theta$ a branched $k$-fold cover of $S^2$ along  $S^0.$
Then the inclusion $A \subset B$ can be described homotopically as follows:

In $S^3 \times [0, 1]$ contract an interval $*\times [0,1]$ for some $* \in
S^3$ to a point. $A$ will be identified with $S^3\times \{0\}.$
Further on $S^3 \times \{1\}$ identify the points that are mapped into the
same point by the suspension of $\theta$.
The part of $B$ coming from $S^3 \times \{1\}$ will be denoted by $B_1$.
The space $B_1$ is a deformation retract of~$B$.

Let us denote by $r$ the retraction $B \longrightarrow  B_1$.
Clearly, its restriction  $r|_A: \ A \longrightarrow  B_1$ is a degree
$k$ map (it is actually the suspension of the  branched covering~$\theta$).
So the inclusion $A \subset B$ induces in the $3$-dimensional homology group $H_3$ (or in $\pi_3$) a multiplication by~$k$.

The proof of the special case (when in $M$ there is a single $2$-cell and a single $3$-cell) is finished.
The general case follows easily taking first the quotient of $sk_2\, M$ by all but one $2$-cell and considering any single $3$-cell.

\small
\bigskip
\noindent
Authors addresses:\\
Atsuko Katanaga\\
School of General Education, Shinshu University\\
3-1-1 Asahi, Matsumoto-shi\\
Nagano 390-8621, Japan\\
katanaga@shinshu-u.ac.jp\\
\\
Andr\'as N\'emethi\\
Alfr\'ed R\'enyi Mathematical Institute\\
Hungarian Academy of Sciences\\
Re\'altanoda u. 13-15\\
H-1053 Budapest, Hungary\\
nemethi@renyi.hu\\
\\
Andr\'as Sz\H{u}cs\\
Department of Analysis, E\"otv\"os University\\
P\'azm\'any P. s\'et\'any I/C\\
H-1117 Budapest, Hungary\\
szucs@cs.elte.hu


\footnotesize

\begin{thebibliography}{88}
\bibitem[Ba]{Ba}
D. Barden: Simply connected five-manifolds. Annals of Math. {\bf 82} , (1965), 365--385.

\bibitem[B]{B}
R. Bott: The stable homotopy of the classical groups,
{\it Proceedings of the National Academy of Sciences} {\bf 43} (1957), 933--935.

\bibitem[E-Sz]{E-Sz}
T. Ekholm, A. Sz\H{u}cs: The group of immersions of homotopy $(4k - 1)$-spheres,
{\it Bull. London Math. Soc} {\bf 38} (2006), 163--176.

\bibitem[H]{H}
M. W. Hirsch: Immersions of manifolds,
{\it Trans. Amer. Math. Soc.} {\bf 93} (1959), 242--276.

\bibitem[Hu]{Hu}
D. Husemoller: {\it Fibre bundles}, McGraw-Hill Book Co., New York--London--Sydney, 1966.

\bibitem[K-N]{K-N}
A. Katanaga; K.  Nakamoto:
The links of $3$-dimensional singularities defined by Brieskorn polynomials,
{\it Math. Nachr.} {\bf 281} (2008), no. 12, 1777--1790.

\bibitem[M-S]{M-S}
J. W. Milnor, J. D. Stasheff: {\it Characteristic classes},
 Princeton University Press, Princeton, 1974.

\bibitem[R-S]{R-S}
C. Rourke, B. Sanderson:
The compression theorem I,
{\it Geometry and Topology} {\bf 5} (2001), 399--429.

\bibitem[P]{P}
U. Pinkall: Regular homotopy classes of immersed surfaces,
{\it Topology} {\bf 24} (1985), no. 4, 421--434.

\bibitem[S1]{S1}
S. Smale: Classification of immersion of spheres in Euclidean space,
{\it Ann. of Math.} {\bf 69} (1959), 327--344.

\bibitem[S2]{S2}
S. Smale: On the structure of $5$-manifolds,
{\it Ann. of Math. (2)} {\bf 75} (1962), 38--46.

\bibitem[T]{T}
R. Thom:
La classification des immersions, {\it S\'eminaire Bourbaki}, 1957.

\end{thebibliography}
\end{document}